\documentclass[11pt]{article}

\usepackage[T1]{fontenc}
\usepackage[latin9]{inputenc}
\usepackage{geometry}
\usepackage{amsmath}
\usepackage{amsthm}
\usepackage{amssymb}
\usepackage{color}
\usepackage{paralist}
\usepackage{csquotes}
\usepackage[unicode=true,pdfusetitle,
 bookmarks=true,bookmarksnumbered=false,bookmarksopen=false,
 breaklinks=false,pdfborder={0 0 0},pdfborderstyle={},backref=false,colorlinks=false]
 {hyperref}
 
\usepackage[square,sort,comma,numbers]{natbib} 
\usepackage{setspace}
\makeatletter

\usepackage[nameinlink,capitalise,noabbrev]{cleveref}
\hypersetup{%
    bookmarksnumbered, bookmarksopen=true, bookmarksopenlevel=1,%
}

\theoremstyle{plain}
\newtheorem{thm}{Theorem}[section]
\crefname{thm}{Theorem}{Theorems}
\theoremstyle{plain}
\newtheorem{lem}[thm]{Lemma}
\crefname{lem}{Lemma}{Lemmas}
\theoremstyle{plain}
\newtheorem{cor}[thm]{Corollary}
\theoremstyle{plain}
\newtheorem*{claim*}{Claim}
\crefname{claim}{Claim}{Claims}
\theoremstyle{definition}

\theoremstyle{plain}
\newtheorem{conjecture}[thm]{Conjecture}
\newtheorem{prop}[thm]{Proposition}
\theoremstyle{definition}

\theoremstyle{definition}

\theoremstyle{plain}
\newtheorem{claim}[thm]{Claim}

\crefformat{equation}{#2(#1)#3}
\crefname{appsec}{Appendix}{Appendices}

\date{}

\usepackage{appendix}

\crefformat{equation}{#2(#1)#3}

\let\originalleft\left
\let\originalright\right
\renewcommand{\left}{\mathopen{}\mathclose\bgroup\originalleft}
\renewcommand{\right}{\aftergroup\egroup\originalright}
\usepackage{verbatim}

\makeatletter
\renewcommand*{\UrlTildeSpecial}{%
  \do\~{%
    \mbox{%
      \fontfamily{ptm}\selectfont
      \textasciitilde
    }%
  }%  
}%    
\let\Url@force@Tilde\UrlTildeSpecial
\makeatother

\makeatother

\renewcommand{\Pr}{\mathbb{P}}

\begin{document}

\title{Dense induced bipartite subgraphs in triangle-free graphs}

\author{Matthew Kwan \thanks{Department of Mathematics, Stanford University, CA 94305. Email: \href{mailto:mattkwan@stanford.edu} {\nolinkurl{mattkwan@stanford.edu}}. Research supported in part by SNSF project 178493.}
\and
Shoham Letzter\thanks{ETH Institute for Theoretical Studies, ETH Zurich, Switzerland. Email: \href{mailto:shoham.letzter@eth-its.ethz.ch} {\nolinkurl{shoham.letzter@math.ethz.ch}}. Research supported by Dr.\ Max R\"ossler, the Walter Haefner Foundation and the ETH Zurich Foundation.}
\and
Benny Sudakov\thanks{Department of Mathematics, ETH Zurich, Switzerland. Email:
\href{mailto:benjamin.sudakov@math.ethz.ch} {\nolinkurl{benjamin.sudakov@math.ethz.ch}}.
Research supported in part by SNSF grant 200021-175573.}
\and
Tuan Tran\thanks{Department of Mathematics, ETH Zurich, Switzerland. Email:
\href{mailto:manh.tran@math.ethz.ch} {\nolinkurl{manh.tran@math.ethz.ch}}.
Research supported by the Humboldt Research Foundation.}}

\maketitle
\global\long\def\cA{\mathcal{A}}
\global\long\def\Ex{\mathbb{E}}
\global\long\def\RR{\mathbb{R}}
\global\long\def\E{\mathbb{E}}
\global\long\def\cF{\mathcal{F}}
\global\long\def\Var{\operatorname{Var}}
\global\long\def\NN{\mathbb{N}}
\global\long\def\eps{\varepsilon}
\global\long\def\one{\boldsymbol{1}}
\global\long\def\range#1{\left[#1\right]}
\global\long\def\Bin{\operatorname{Bin}}
\global\long\def\cT{\mathcal{T}}
\global\long\def\floor#1{\left\lfloor #1\right\rfloor }
\global\long\def\ceil#1{\left\lceil #1\right\rceil }

\begin{abstract}
    The problem of finding dense induced bipartite subgraphs in $H$-free graphs has a long history, and was posed 30 years ago by Erd\H os, Faudree, Pach and Spencer. In this paper, we obtain several results in this direction. First we prove that any $H$-free graph with minimum degree at least $d$ contains an induced bipartite subgraph of minimum degree at least $c_H \log d/\log \log d$, thus nearly confirming one and proving another conjecture of Esperet, Kang and Thomass\'e. Complementing this result, we further obtain optimal bounds for this problem in the case of dense triangle-free graphs, and we also answer a question of Erd\H os, Janson, \L uczak and Spencer.     
\end{abstract}

\section{Introduction}
    The Max Cut problem seeks to find the largest bipartite subgraph of a graph $G$. This problem has been studied extensively both from the algorithmic perspective in computer science and optimisation, and from the extremal perspective in combinatorics. Let $b(G)$ denote the size of the largest bipartite subgraph of a graph $G$ with $n$ vertices and $e$ edges. The extremal aspect of the Max Cut problem asks for lower bounds on $b(G)$ in terms of $n$ and $e$. A classical result of Erd\H os from 1965 (see~\cite{Erd65}) states that $b(G) \ge e/2$, and the complete graph on $n$ vertices shows that the constant $1/2$ in this bound is asymptotically tight. As such, much of the focus has been on determining the error term $b(G)-e/2$. For instance, Erdwards \cite{Erdw73} proved $b(G) - e/2 \ge (\sqrt{8e+1}-1)/8$, and noted that equality holds for complete graphs. We refer the reader to \cite{Alo96,BolSco,CFKS} and the references therein for further results in this direction. Following these works, there has been a lot of research concerning maximum bipartite subgraphs of various restricted classes of graphs. One class which has drawn most of the attention is $H$-free graphs, with work dating back to Erd\H os and Lov\'asz (see \cite{Erd76}). Despite great effort, there are only a few graphs $H$ for which we know (asymptotically) optimal lower bounds on $b(G)-e/2$, as $G$ ranges over all $H$-free graph on $e$ edges (see \cite{Alo96,ABKS,AKS2}).

    In the dense case (i.e.\ $e=\Omega(n^2)$), there is a longstanding conjecture of Erd\H os \cite{Erd76} which posits that given any $n$-vertex triangle-free graph $G$, one can delete at most $n^2/25$ edges to make it bipartite (see~\cite{EFPS,EGS} for partial results and \cite{Sud07} for the analogous question in the $H$-free setting). In order to make progress on Erd\H os's problem, Erd\H os, Faudree, Pach and Spencer~\cite{EFPS} initiated the study of the size of the largest {\em induced} bipartite subgraphs in triangle-free graphs, and proved some results in this direction. Similar results were also the main subject of a later paper by Erd\H os, Janson, \L uczak and Spencer~\cite{EJLS}.

    In addition to maximising the number of edges, it is also very natural to %find 
    study induced bipartite subgraphs of large minimum degree. This more local approach was taken recently by Esperet, Kang and Thomass\'e in \cite{EKT}, where they made a number of intriguing conjectures.

    \begin{conjecture} \label{conj:log}
    $\hfill$
		\begin{enumerate}[\rm (i)]
            \item \label{itm:conj-general}
                Any graph with minimum degree at least $d$ contains either a complete subgraph on $\omega_{d\to \infty}(1)$ vertices or an induced bipartite subgraph with minimum degree $\omega_{d\to \infty}(1)$.
            \item \label{itm:conj-triangle-free}
                There is a constant $c>0$ such that any triangle-free graph with minimum degree at least $d$ contains an induced bipartite subgraph of minimum degree at least $c \log d$.
		\end{enumerate}
    \end{conjecture}

	We remark that Esperet, Kang and Thomass\'e were motivated by the concept of {\em separation choosability}; see \cite{EKT} for more details.

    Part \eqref{itm:conj-triangle-free} is a quantitative version of a special case of Part \eqref{itm:conj-general}. Esperet, Kang and Thomass\'e \cite[Theorem 3.8]{EKT} showed that the second part, if true, is sharp up to a constant factor, by considering an appropriate binomial random graph. On the other hand, they mentioned that proving the existence of an induced bipartite subgraph with minimum degree at least $3$ seems difficult even in the case of graphs with large minimum degree and large girth (see~\cite[Conjecture 1.5]{EKT}). As our main result, we prove a slight weakening of Part \eqref{itm:conj-triangle-free} of \cref{conj:log} and settle Part \eqref{itm:conj-general} in a strong form.

    \begin{thm} \label{thm:log}
        For every graph $H$, there is a constant $c_H>0$ such that every $H$-free graph with minimum degree at least $d$ contains an induced bipartite subgraph of minimum degree at least $c_H \log d/\log \log d$.
    \end{thm}
    
    We remark that this result is essentially tight: if $H$ is a graph that is not a forest (note that the statement is vacuous when $H$ is a forest, because the minimum degree condition forces the appearance of any forest on $d$ vertices), one can show using random graphs that there is an $H$-free graph with minimum degree $d$ which does not contain an induced bipartite subgraph with minimum degree at least $c_H \log d$ for some constant $c_H$ and for every large enough $d$. We show how to obtain such a construction for triangle-free graphs (see \cref{sec:probabilistic}), but it is not hard to obtain a construction in general.
    We remark also that for triangle-free graphs which are reasonably dense (i.e.\ $d=n^{\Omega(1)}$), one can get rid of the $\log \log d$ factor in \cref{thm:log} (see \cref{sec:poly-large} for details). We first prove \cref{thm:log} for triangle-free graphs in \cref{sec:log}. We then prove a reduction to the triangle-free case in \cref{sec:reduction}.  Although we make no effort to optimise the constant $c_H$, we remark that our proof gives a bound of the form $c_H \ge \varepsilon^{|H|}$ for some constant $\varepsilon>0$.

    Next, we turn our attention to the more refined problem of finding induced bipartite subgraphs with large minimum degree in terms of both the number of vertices and the minimum degree of a host triangle-free graph. Let $g(n,d)$ be the maximum $g$ such that every triangle-free graph with $n$ vertices and minimum degree at least $d$ contains an induced bipartite subgraph with minimum degree at least $g$. We observe an interesting `phase transition' behaviour of this function.

    \begin{thm}\label{thm:refined-degree}
        With $g(n,d)$ as defined above,
        \begin{enumerate}[\rm (i)]
            \item \label{itm:medium-deg}
                for $d \le \sqrt n$ we have $\Omega(\log d/\log \log d)\le g(n,d)\le O\left(\log d\right)$;
            \item \label{itm:large-deg}
                for $\sqrt n < d \le n/2$ we have $d^2/(2n)\le g(n,d)\le O((d^2\log d)/n)$.
            \newcounter{temp}
            \setcounter{temp}{\value{enumi}}
        \end{enumerate}
		Furthermore,
        \begin{enumerate}[\rm (i)]
            \setcounter{enumi}{\value{temp}}
            \item \label{itm:sharp-medium-deg}
                for $n^{\Omega(1)}\le d \le \sqrt n$ we have $g(n,d)=\Theta\left(\log d\right)$;
            \item \label{itm:shart-large-deg}
                for $n^{2/3} < d \le n/2$ we have $g(n,d)=\Theta(d^2/n)$.
        \end{enumerate}
    \end{thm}

    Part \eqref{itm:large-deg} of \cref{thm:refined-degree} was also proved independently by Cames van Batenburg, de Joannis de Verclos, Kang and Pirot~\cite[Theorem 1.2]{BVKP}, though our proofs of both the lower and upper bounds are %slightly 
    different to theirs. We prove all the lower bounds for \cref{thm:refined-degree} in \cref{sec:lower}, and we prove the upper bounds in \cref{sec:upper}.

    Finally, let us return to the work of Erd\H os, Faudree, Pach and Spencer~\cite{EFPS}, and make some remarks. These authors defined $f(n,m)$ to be the maximum $f$ such that every triangle-free graph with $n$ vertices and at least $m$ edges contains an induced bipartite subgraph with at least $f$ edges, and they studied the approximate behaviour of $f(n,m)$ as $n$ and $m$ vary. They showed $\Omega(m^{1/3})\le f(n,m)\le O(m^{1/3}\log^2m)$ for $m \le n^{3/2}$, and $\Omega(m^3n^{-4})\le f(n,m)\le O(m^3n^{-4}\log^2(n^2/m))$ if $m\ge n^{3/2}$. By leveraging some results concerning discrepancy and chromatic number of triangle-free graphs, we can find the correct order of magnitude of $f(n,m)$ for all $n,m$.

    \begin{prop}\label{thm:ManyEdges}
        With $f(n,m)$ as defined above,
        \begin{enumerate}[\rm (i)]
            \item \label{itm:many-edges-sparse}
                for $m \le n^{3/2}\sqrt{\log n}$ we have $f(n,m)=\Theta\left(m^{1/3}\log^{4/3}m\right)$;
            \item \label{itm:many-edges-dense}
                for $m \ge n^{3/2}\sqrt{\log n}$ we have $f(n,m)=\Theta(m^3/n^4)$.
        \end{enumerate}
    \end{prop}

	\cref{thm:ManyEdges} resolves a problem raised in a paper of Erd\H os, Janson, {\L}uczak and Spencer~\cite{EJLS}. We present its short proof in \cref{sec:many-edges}.

    \paragraph{Notation.}%
        We use standard asymptotic notation throughout the paper. For functions $f=f\left(n\right)$ and $g=g\left(n\right)$ we write $f=O\left(g\right)$ to mean there is a constant $c$ such that $\left|f\right|\le c\left|g\right|$, we write $f=\Omega\left(g\right)$ to mean there is a constant $c>0$ such that $f\ge c\left|g\right|$ for sufficiently large $n$, we write $f=\Theta\left(g\right)$ to mean that $f=O\left(g\right)$ and $f=\Omega\left(g\right)$, and we write $f=o\left(g\right)$ or $g=\omega\left(f\right)$ to mean that $f/g\to0$ as $n\to\infty$. All asymptotics are as $n\to\infty$ unless specified otherwise (specifically, notation of the form $o_{k}\left(1\right)$ indicates that asymptotics are as $k\to\infty$). Also, we write `w.h.p.'\ (standing for `with high probability') to indicate that an event holds with probability $1-o(1)$.

\section{Induced bipartite subgraphs with many edges}\label{sec:many-edges}
    In this section we will prove \cref{thm:ManyEdges}. The lower bound in \eqref{itm:many-edges-dense} was already proved in~\cite{EFPS,EJLS}. For the lower bound in \eqref{itm:many-edges-sparse}, we recall the following easy observation, which appears in the proof of~\cite[Theorem~4]{EFPS}.

    \begin{lem}\label{lem:EFPS}
        Any graph $G$ has an induced bipartite subgraph with at least $e(G)/\binom{\chi(G)}{2}$ edges.
    \end{lem}

    The desired lower bound in \eqref{itm:many-edges-sparse} is then an immediate consequence of the following bound on the chromatic number of triangle-free graphs, which seems to have appeared independently in~\cite{GimTho, Nil00, PolTuz}.

    \begin{thm} \label{thm:chromatic}
        Let $G$ be a triangle-free graph with $m$ edges. Then $\chi(G)=O\left(\frac{m^{1/3}}{\log^{2/3} m}\right)$.
    \end{thm}

    It remains to prove the upper bounds in \eqref{itm:many-edges-sparse} and \eqref{itm:many-edges-dense}. For both of these we will take advantage of the following theorem proved by Guo and Warnke~\cite[Theorem 4]{GuoWar} resulting from analysis of a certain random triangle-free process.

    \begin{thm} \label{thm:GuoWarnke}
        There exist $n_{0},c,\beta>0$ such that for all $n\ge n_{0}$, there is an $n$-vertex triangle-free graph $G_{n}$ with the following property. For any pair of disjoint vertex-sets $A,B\subseteq V\left(G_{n}\right)$ with $\left|A\right|=\left|B\right|=\ceil{c\sqrt{n\log n}}$ we have
        \begin{equation}
            \beta \sqrt{\frac{\log n}{n}} \, \left|A\right|\left|B\right|
            \le e(G_{n}\left[A,B\right])
            \le 2\beta\sqrt{\frac{\log n}{n}} \, \left|A\right|\left|B\right|.\label{eq:disc}
        \end{equation}
    \end{thm}

    The reason \cref{thm:GuoWarnke} is useful for us is as follows.
    \begin{lem} \label{lem:GW-deduction}
        Consider a graph $G_{n}$ as in \cref{thm:GuoWarnke}. Then
        \begin{enumerate} [\rm (i)]
            \item \label{itm:GW-edges}
                $e\left(G_{n}\right)=\Theta\left(n^{3/2}\sqrt{\log n}\right)$, and;
            \item \label{itm:GW-bip}
                every induced bipartite subgraph in $G_{n}$ has $O(\sqrt{n \log n})$ vertices and  $O\left(\sqrt{n}\log^{3/2}n\right)$ edges.
        \end{enumerate}
    \end{lem}

    \begin{proof}
        For \eqref{itm:GW-edges}, note that in fact we have
        \[
            \beta\sqrt{\left(\log n\right)/n}\binom{n}{2}\le e\left(G_{n}\right)\le2\beta\sqrt{\left(\log n\right)/n}\binom{n}{2}.
        \]
        Indeed, if $e\left(G_{n}\right)$ fell outside this range, then the average value of $e(G_{n}\left[A,B\right])$, over all disjoint pairs $(A,B)$ of vertex subsets of size $\ceil{c\sqrt{n\log n}}$, would fall outside the range in \cref{eq:disc}.

        For \eqref{itm:GW-bip}, first note that every set of $2\ceil{c\sqrt{n\log n}}$ vertices contains at least one edge, because it can be partitioned into two sets $A,B$, which by \cref{eq:disc} must have an edge between them. That is to say, $\alpha\left(G_{n}\right)=O\left(\sqrt{n\log n}\right)$. But for any induced bipartite subgraph, the two parts of that subgraph are each independent sets, and therefore such a subgraph consists of at most $2\alpha\left(G_{n}\right)=O\left(\sqrt{n\log n}\right)$ vertices. Then, the number of edges between the two parts is
        \[
            O\left(\sqrt{\log n/n}\left(\sqrt{n\log n}\right)^{2}\right)=O\left(\sqrt{n}\log^{3/2}n\right).
        \]
        (We can arbitrarily enlarge $A,B$ if necessary, in order to use \cref{eq:disc}).
    \end{proof}

    We deduce the following corollary by taking blowups of the graphs in \cref{lem:GW-deduction}.

    \begin{cor} \label{cor:GW-blowup}
        For any $k \in \NN$, there is an $n$-vertex triangle-free graph $G_{n}^{\left(k\right)}$ with the following properties.
        \begin{enumerate}[\rm (i)]
            \item \label{itm:cor-GW-edges}
                $e\left(G_{n}^{\left(k\right)}\right)=\Theta\left(n^{3/2}\sqrt{k\log (n/k)}\right)$, and;
            \item \label{itm:cor-GW-bip}
                every induced bipartite subgraph in $G_{n}^{\left(k\right)}$ has $O\left(k^{3/2}\sqrt{n}\log^{3/2}(n/k)\right)$ edges.
        \end{enumerate}
    \end{cor}

    \begin{proof}
		We note that the corollary holds trivially when $k = \Omega(n)$ (take, say, $G_{n}^{\left(k\right)}$ to be the complete bipartite graph $K_{n/2, n/2}$). Thus, we may assume that $n/k \ge n_0$, where $n_0$ is as in \Cref{thm:GuoWarnke}.

        Let $G_{\ceil{n/k}}$ be as in \cref{lem:GW-deduction}, and define the $n$-vertex graph $G_{n}^{\left(k\right)}$ as follows. The vertex set of $G_{n}^{\left(k\right)}$ consists of the disjoint union of $\ceil{n/k}$ sets $S_{v}$ each indexed by a vertex $v$ of $G_{\ceil{n/k}}$ and having sizes as equal as possible (so each $S_{v}$ has size $\Theta(k)$). For every edge $uv \in E\left(G_{\ceil{n/k}}\right)$, we put a complete bipartite graph between each vertex of $S_{u}$ and each vertex of $S_{v}$. This graph $G_{n}^{\left(k\right)}$ is called a \emph{blowup} of $G_{\ceil{n/k}}$. Note that $G_{n}^{\left(k\right)}$ is triangle-free, and that $e\left(G_{n}^{\left(k\right)}\right)$ exceeds $e\left(G_{\ceil{n/k}}\right)=\Theta\left(\left(n/k\right)^{3/2}\sqrt{\log (n/k)}\right)$ by a factor of $\Theta(k^2)$, implying \eqref{itm:cor-GW-edges}.

        For \eqref{itm:cor-GW-bip}, note that an induced bipartite subgraph of $G_{n}^{\left(k\right)}$ can only be obtained by `blowing up' an induced bipartite subgraph of $G_{\ceil{n/k}}$. To be more precise, for any induced bipartite subgraph $H$ in $G_{n}^{\left(k\right)}$, consider the set $W$ of all $v\in V\left(G_{\ceil{n/k}}\right)$ such that a vertex in $S_{v}$ appears in $H$. Then $W$ induces a bipartite subgraph of $G_{\ceil{n/k}}$. It follows that every induced bipartite subgraph in $G_{n}^{\left(k\right)}$ has at most $O\left(k\sqrt{\left(n/k\right)\log (n/k)}\right)=O\left(\sqrt{kn\log (n/k)}\right)$ vertices, and at most $O\left(k^{2}\sqrt{n/k}\log^{3/2}(n/k)\right)=O\left(k^{3/2}\sqrt{n}\log^{3/2}(n/k)\right)$ edges.
    \end{proof}
    
    Now we are ready to prove the upper bounds in \cref{thm:ManyEdges}. For every $n$, let $G_n$ be as in \cref{lem:GW-deduction}, so $e\left(G_{n}\right) = \Theta\left(n^{3/2}\sqrt{\log n}\right)$. We will show that if $m\le e\left(G_{n}\right)$ then $f\left(n,m\right)=\Theta\left(m^{1/3}\log^{4/3}m\right)$, and if $m>e\left(G_{n}\right)$ then $f\left(n,m\right)=\Theta\left(m^{3}/n^{4}\right)$.

    First, suppose that $m\le e\left(G_{n}\right)$. Let $n'$ be the least integer such that $m \le e\left(G_{n'}\right)$. It follows that $n' \le n$ and $m=\Theta\left(e\left(G_{n'}\right)\right)=\Theta\left(\left(n'\right)^{3/2}\sqrt{\log n'}\right)$, so $n'=\Theta\left(m^{2/3}/\log^{1/3}m\right)$. If necessary, add some isolated vertices to $G_{n'}$ to obtain a graph $G$ with exactly $n$ vertices and at least $m$ edges. Now, consider any induced bipartite subgraph in $G$, and discard its isolated vertices. By \cref{lem:GW-deduction}, the number of edges in this bipartite subgraph is at most 
    \[
        O\left(\sqrt{n'}\log^{3/2}n'\right)=O\left(m^{1/3}\log^{4/3}m\right),
    \]
    as desired.

    Next, suppose that $m>e\left(G_{n}\right)$. Let $k$ be the maximum integer such that $m \ge e\left(G_{n}^{\left(k\right)}\right)$. Then $m = \Theta\left(e\left(G_{n}^{\left(k\right)}\right)\right) = \Theta\left(n^{3/2}\sqrt{k\log (n/k)}\right)$, so $k \log (n/k) = \Theta\left(m^{2}/n^3)\right)$. By \cref{cor:GW-blowup}, the number of edges in every induced bipartite subgraph in $G_{n}^{\left(k\right)}$ is at most 
    \[
        O\left(k^{3/2}\sqrt{n}\log^{3/2}(n/k)\right)=O\left(m^{3}/n^{4}\right),
    \]
    as desired. This completes the proof of \cref{thm:ManyEdges}.

\section{Lower bounds on minimum degree}\label{sec:lower}

In this section we prove the lower bounds in \cref{thm:refined-degree}. The lower bound for \cref{thm:refined-degree}~\eqref{itm:medium-deg} (which also proves \cref{thm:log} for triangle-free graphs) appears in \cref{sec:log}, the lower bound for \eqref{itm:large-deg} and \eqref{itm:shart-large-deg} appears in \cref{sec:dense}, and the lower bound for \eqref{itm:sharp-medium-deg} appears in \cref{sec:poly-large}.

	\subsection{The sparse regime}\label{sec:log}

        Here we prove \cref{thm:log} for triangle-free graphs, which also provides the lower bound in \cref{thm:refined-degree}~(i). 
    
		Let $G$ be a triangle-free graph with minimum degree at least $d$. We may assume that $G$ is minimal with respect to the minimum degree assumption, namely that every proper induced subgraph of $G$ has minimum degree less than $d$. Under this assumption, note that $G$ is $d$-degenerate. 
		Indeed, by assumption every proper induced subgraph of $G$ has a vertex of degree at most $d$. We can therefore order the vertices of $G$ from left to right in such a way that every vertex has at most $d$ neighbours to its right. Let $N^{+}\left(v\right)$ be the set of neighbours of $v$ to its right, and let $\ell = \floor{\log d / \log \log d}$. Throughout this section we assume that $d$ is sufficiently large. The heart of the proof of \cref{thm:log} for triangle-free graphs is the following claim.
        
        \begin{claim} \label{claim:X-Y-Z}
            There exist two non-empty disjoint sets $X, Y \subseteq V(G)$ with the following properties.
            \begin{enumerate}[\rm (i)]
                \item \label{itm:min-deg}
                    every vertex in $Y$ has at least $\ell$ neighbours in $X$;
                \item \label{itm:size-Z}
                    there are at most $(\ell/7) |Y|$ edges between $X$ and $Y$ that are incident to a vertex $x$ in $X$ with $|N^+(x) \cap X| \ge 4$;
                \item \label{itm:size-X}
                    $|X| < 3 |Y|$;
                \item \label{itm:size-e-Y}
                    $e(Y) < 3 |Y|$.
            \end{enumerate}
        \end{claim}

    	Before proving \cref{claim:X-Y-Z}, we will show how it implies \cref{thm:log} for triangle-free graphs.

		\begin{proof}[Proof of \cref{thm:log} for triangle-free graphs]
			Let $X$ and $Y$ be the sets given by \cref{claim:X-Y-Z}. Let $X_0$ be the set of vertices $x \in X$ for which $|N^+(x) \cap X| \ge 4$. The graph $G[X \setminus X_0]$ is $3$-degenerate, hence it is $4$-colourable. Let $\{X_1, \ldots, X_4\}$ be a partition of $X \setminus X_0$ into four independent sets. Let $Y_0$ be the set of vertices in $Y$ that send at least $3\ell/7$ edges into $X_0$. The number of edges between $X_0$ and $Y_0$ is at least $(3\ell / 7) |Y_0|$, and at most $(\ell / 7)|Y|$ (by Property \eqref{itm:size-Z} of the claim). So $|Y_0| \le (1/3)|Y|$. Note that the number of edges in $Y\setminus Y_0$ is at most $e(Y) \le 3|Y|  \le (9/2)|Y\setminus Y_0|$, due to Property \eqref{itm:size-e-Y}. In other words, the average degree of $G[Y\setminus Y_0]$ is at most $9$. Thus by Tur\'an's theorem \cite{Tur41}, $Y\setminus Y_0$ contains an independent set $Y'$ of size at least $(1/10)|Y\setminus Y_0| \ge (1/15)|Y|$. To prove \cref{thm:log}, it suffices to show that $G[X_i,Y']$ has average degree at least $\Omega(\ell)$ for some $i\in [4]$. 
			
			From Property \eqref{itm:min-deg} and the definition of $Y_0$, we see that every vertex in $Y\setminus Y_0$ has at least $4\ell/7$ neighbours in $X\setminus X_0$. Moreover, as 
			\[
			\left(\tfrac{1}{4}|X\setminus X_0| +|Y'|\right) \cdot \sum_{1\le i \le 4}e(X_i,Y')=\tfrac{1}{4}\sum_{1\le i\le 4} (|X_i|+|Y'|)\cdot e(X\setminus X_0,Y'),
			\]
			there exists $i\in [4]$ such that $\left(\frac{1}{4}|X\setminus X_0|+|Y'|\right) \cdot e(X_i,Y')\ge \frac{1}{4}(|X_i|+|Y'|)\cdot e(X\setminus X_0,Y')$. Therefore, the average degree of $G[X_i,Y']$ is
			\[
				\frac{2 e(X_i,Y')}{|X_i|+|Y'|} \ge \frac{e(X\setminus X_0,Y')}{(1/2)|X\setminus X_0|+2|Y'|}\ge \frac{(4\ell/7)|Y'|}{(1/2)|X\setminus X_0|+2|Y'|} \ge (8/343)\ell.
			\]
			Here the last inequality holds since $|Y'| \ge (1/15)|Y|$ and $\frac{(4\ell/7)|Y'|}{(1/2)|X\setminus X_0|+2|Y'|}$ is an increasing function in $|Y'|$, and since $|X\setminus X_0| \le 3|Y|$ (by Property \eqref{itm:size-X}). Thus $G$ contains an induced bipartite graph with minimum degree at least $(8/343)\ell=\Omega(\log d/\log \log d)$.
		\end{proof}

		The rest of this section is devoted to the proof of \cref{claim:X-Y-Z}. Let $p=1/d$. For every vertex $u \in V(G)$, define $p_u$ to satisfy $(1 - p) \cdot \Pr[\Bin(d(u), p) \ge \ell] \cdot p_u = p$. We define two random sets $X$ and $Y$ as follows. First, each vertex $x\in V(G)$ is included in $X$ with probability $p$, independently. Second, for every vertex $y\in V(G)$, we put it in $Y$ if it is not in $X$, it has at least $\ell$ neighbours in $X$, and a biased coin flip, whose probability of heads is $p_y$, turns out to be heads. (These coin flips are independent for each $y\in V(G)$.) Let $Z$ be the set of edges $xy$ where $x \in X$, $y \in Y$ and $|N^+(x) \cap X| \ge 4$. In order to prove \cref{claim:X-Y-Z}, we prove the following claim.

        \begin{claim} \label{claim:expectations}
            $\Ex \left[ |Y| - (1/3)|X| - (1/3)e(Y) - (7/\ell)|Z| \right]>0$.
        \end{claim}

		Before proving \cref{claim:expectations}, we show that it implies \cref{claim:X-Y-Z}.

		\begin{proof}[Proof of \cref{claim:X-Y-Z} given \cref{claim:expectations}]
			By \cref{claim:expectations}, there is a choice of sets $X$ and $Y$ as above, with
			\[
				|Y| - (1/3)|X| - (1/3)e(Y) - (7/\ell)|Z| > 0.
			\]
			In particular, $X$ and $Y$ are disjoint; every vertex in $Y$ sends at least $\ell$ edges into $X$; $|Z| < (\ell / 7) |Y|$, i.e.\ there are at most $(\ell / 7) |Y|$ edges that are incident with a vertex $x$ in $X$ with $|N^+(x) \cap X| \ge 4$; $|X| < 3|Y|$; and $e(Y) < 3 |Y|$.
			\cref{claim:X-Y-Z} follows.
		\end{proof}

		To prove \cref{claim:expectations} we shall need the following routine estimates concerning the binomial distribution. The reader may wish to skip their verification and go straight to the proof of \cref{claim:expectations}, given after \cref{lem:binomial}. We remark that the need to satisfy \eqref{itm:binom-one} is the bottleneck in the choice of the value $\ell$.

        \begin{lem}\label{lem:binomial} 
            We have
            \begin{enumerate}[\rm (i)]
                \item \label{itm:binom-one}
                    $(1 - p) \cdot \Pr[ \Bin(d, p) \ge \ell ] \ge p$, 
                \item \label{itm:binom-two}
                    $\Pr[\Bin(m-1, p) \ge \ell-1] \le (\ell+1)\cdot \Pr[\Bin(m, p) \ge \ell]$ for every $m \ge d$.
            \end{enumerate}
        \end{lem}

        \begin{proof}
            For Part \eqref{itm:binom-one} we bound
            \begin{align*}
                (1 - p) \cdot \Pr[ \Bin(d, p) \ge \ell ] 
                & \ge \binom{d}{\ell} p^{\ell} (1-p)^{d-\ell+1} 
                \ge \left( \frac{pd}{\ell}\right)^\ell e^{-2p(d-\ell+1)} \\
                & \ge \ell^{-\ell} e^{-2}=e^{-\ell \log \ell -2}\\
                &\ge e^{-\frac{\log d}{\log \log d}(\log \log d - \log \log \log d) - 2} \\
                & \ge e^{- \log d} = p.
            \end{align*}
            Here we used the inequalities $\binom{k}{t} \ge (k/t)^t$, which holds for $1 \le t \le k$, and $1-x \ge e^{-2x}$, which holds for $0 < x < 1/2$. For Part \eqref{itm:binom-two} we estimate
            \begin{align*}
            \frac{\Pr[\Bin(m-1, p) \ge \ell-1]}{\Pr[\Bin(m, p) \ge \ell]} 
                    & = \frac{\Pr[\Bin(m-1, p) = \ell-1]}{\Pr[\Bin(m, p) \ge \ell]} + \frac{\Pr[\Bin(m-1, p) \ge \ell]}{\Pr[\Bin(m, p) \ge \ell]} \\
                    & \le \frac{\Pr[\Bin(m-1, p) = \ell-1]}{\Pr[\Bin(m, p) = \ell]} + 1 \\
                    & = \frac{\binom{m-1}{\ell-1} p^{\ell-1} (1-p)^{m-\ell}}{\binom{m}{\ell} p^\ell (1-p)^{m-\ell}} + 1 \\
                    & = \frac{\ell}{mp} + 1 \le \ell+1. \qedhere
            \end{align*}
        \end{proof}    
        
        \begin{proof}[Proof of \cref{claim:expectations}]
            We shall estimate the expectations of the four quantities $|X|,|Y|, e(Y),|Z|$ separately, and show that they satisfy the desired inequality. Note that, by \cref{lem:binomial}~\eqref{itm:binom-one} and the choice of $p_u$, we have $0 < p_u \le 1$. In addition, the probability that a vertex $y$ is in $Y$ is $(1 - p) \cdot \Pr[\Bin(d(y), p) \ge \ell] \cdot p_y = p$. 
             
            Since each vertex is in $X$ with probability $p$, and similarly each vertex is in $Y$ with probability $p$, we have
            \begin{equation} \label{eqn:exp-size-X}
                \Ex[|X|] = \Ex[|Y|] = np = n/d.
            \end{equation}
            
            Let $uv$ be an edge. Then the probability of the event $u, v \in Y$ is 
            \begin{align*}
                &p_u \cdot p_v \cdot \Pr\!\big[|N(u) \cap X|, |N(v) \cap X| \ge \ell \text{ and } u, v \notin X\big] \\
                =\,\,&p_u \cdot p_v \cdot (1 - p)^2 \cdot \Pr\!\big[|(N(u) \setminus \{v\}) \cap X|, |(N(v) \setminus \{u\}) \cap X| \ge \ell\big] \\
                =\,\,&p_u \cdot p_v \cdot (1 - p)^2 \cdot \Pr\!\big[|(N(u) \setminus \{v\}) \cap X| \ge \ell\big] \cdot \Pr\!\big[|(N(v) \setminus \{u\}) \cap X| \ge \ell\big] \\
                \le \,\, &p_u \cdot p_v \cdot (1 - p)^2 \cdot \Pr[\Bin(d(u), p) \ge \ell] \cdot \Pr[\Bin(d(v), p) \ge \ell]
                = p^2.
            \end{align*}
            In the second equality we used the fact that the sets $N(u) \setminus \{v\}$ and $N(v) \setminus \{u\}$ are disjoint (as $G$ is triangle-free). This implies that the events $|(N(u) \setminus \{v\}) \cap X| \ge \ell$ and $|(N(v) \setminus \{u\}) \cap X| \ge \ell$ are independent. As $G$ is $d$-degenerate, it follows that
            \begin{equation} \label{eqn:exp-edges-Y}
                \Ex[e(Y)] \le e(G) p^2 \le nd p^2 = n/d.
            \end{equation}
            
            Given an edge $xy$, the probability that $xy \in Z$, with $x \in X, y \in Y$, is
            \begin{align*}
                & p_y \cdot \Pr\!\big[ x \in X,\, |N^+(x) \cap X| \ge 4, \, y \notin X, \, |N(y) \cap X| \ge \ell \big] \\
                =\, \, & p_y \cdot p  \cdot \Pr\!\big[|(N^+(x) \setminus \{y\}) \cap X| \ge 4 \big] \cdot (1 - p) \cdot \Pr\!\big[|(N(y)     \setminus \{x\}) \cap X| \ge \ell-1\big] \\
                \le\, \, & p_y \cdot p  \cdot \binom{d}{4} p^4 \cdot (1 - p) \cdot \Pr\!\big[\Bin(d(y) - 1, p) \ge \ell-1\big] \\
                \le\, \,& p_y \cdot p   \cdot \frac{(pd)^4}{4!} \cdot (1 - p) \cdot \Pr[\Bin(d(y), p) \ge \ell] \cdot (\ell+1) \\
                =\, \,& p^2 \cdot \frac{1}{24} \cdot (\ell+1) \, \le \, p^2 \ell/23.
            \end{align*}
            In the first equality we again used the assumption that $G$ is triangle-free to deduce independence of events that depend on $N(x) \setminus \{y\}$ and $N(y) \setminus \{x\}$. In the second inequality we used the inequality $\binom{k}{t} \le k^t/t!$ and \cref{lem:binomial}~\eqref{itm:binom-two} and in the last equality we used the definition of $p_y$. We deduce that
            \begin{align} \label{eqn:exp-size-Z}
                \Ex[|Z|] \le e(G) \cdot  (p^2\ell/23) \le nd \cdot (p^2\ell/23) = n\ell/(23d).
            \end{align}
            
            Finally, it follows from \eqref{eqn:exp-size-X}, \eqref{eqn:exp-edges-Y} and \eqref{eqn:exp-size-Z} that
            \[
             \Ex \left[ |Y| - (1/3)|X| - (1/3)e(Y) - (7/\ell)|Z| \right] > n/d- n/(3d)-n/(3d)-n/(3d) = 0,
            \]
            finishing the proof of \cref{claim:expectations}.
        \end{proof}

	\subsection{The dense regime}\label{sec:dense}

		In this subsection we will provide two different proofs of the fact that $g(n,d) = \Omega(d^2/n)$. We note that a different proof of the bound $g(n, d) \ge d^2 / 2n$ (which is the bound obtained in the first proof) is given in \cite{BVKP}.
		This will establish the lower bound in parts \eqref{itm:large-deg} and \eqref{itm:shart-large-deg} of \cref{thm:refined-degree}. Let $G$ be a triangle-free graph on $n$ vertices with minimum degree at least $d$.

		\begin{proof}[First proof]
			For each vertex $v\in V(G)$, we fix a subset $A_v\subseteq N(v)$ of size $d$. Let $X=\{x_1,x_2\}$ be a uniformly random pair of vertices of $V(G)$, and define a vertex set $Y$ as follows: $v\in V(G)$ is included in $Y$ if and only if $A_v\cap X \ne \emptyset$. Since $G$ is triangle-free, $G[Y]$ is a bipartite graph with bipartition $\{N(x_1)\cap Y, (N(x_2)\setminus N(x_1))\cap Y\}$. Thus, in order to prove the theorem, it suffices to show that  $G[Y]$ has average degree at least $d^2/n$ (and therefore has an induced subgraph of minimum degree at least $d^2/(2n)$) for some choice of $X$.

			For each vertex $v\in V(G)$, we have $\Pr(v\in Y)=1-\binom{n-d}{2}/\binom{n}{2}<\frac{2d}{n}$. Given an edge $uv \in E(G)$, we have $N(u)\cap N(v)=\emptyset$ as $G$ is triangle-free, and so $\Pr(u,v \in Y)=\frac{|A_u||A_v|}{\binom{n}{2}}=\frac{d^2}{\binom{n}{2}}> \frac{2d^2}{n^2}$. Therefore, using the fact that $e(G) \ge nd/2$, we obtain
			\[
				\Ex\left[e(Y)-\frac{d^2}{2n}|Y|\right] > e(G)\cdot \frac{2d^2}{n^2}-\frac{d^2}{2n}\cdot n \cdot \frac{2d}{n} \ge 0.
			\]
			Hence $e(Y)>\frac{d^2}{2n}|Y|$ for some choice of $X,$ as desired.
		\end{proof}

		\begin{proof}[Second proof, with a slightly worse bound]
			Since the statement holds trivially for $d\le 2\sqrt{n}$, we may assume that $d>2\sqrt{n}$. Given $uv \in E(G)$, let $c(u,v)$ be the number of 4-cycles passing through $\{u,v\}$. Since $G$ is triangle-free, we see that $N(u)\setminus \{v\}$ and $N(v)\setminus\{u\}$ are two disjoint independent sets, and that there are exactly $c(u,v)$ edges between $N(u)\setminus \{v\}$ and $N(v)\setminus\{u\}$. Hence the subgraph of $G$ induced by $\left(N(u)\setminus\{v\}\right) \cup \left(N(v)\setminus \{u\}\right)$ is a bipartite graph with average degree at least $2c(u,v)/(d(u)+d(v)-2)$, and thus has an induced bipartite subgraph with minimum degree at least $c(u,v)/(d(u)+d(v)-2)$. In order to prove that $g(n,d)\ge d^2/(4 n)$, it thus suffices to show that $q \ge d^2/(4n)$, where
			\[
				q:=\max_{uv \in E(G)}\frac{c(u,v)}{d(u)+d(v)-2}.
			\]
			From the definition of $q$, we see that $q\cdot \sum_{uv \in E(G)} (d(u)+d(v)-2) \ge \sum_{uv \in E(G)} c(u,v)$. Letting $c_4$ denote the number of 4-cycles in $G$, we then obtain 
			\[
				q \ge \frac{\sum_{uv \in E(G)} c(u,v)}{\sum_{uv \in E(G)} (d(u)+d(v)-2)}=\frac{2c_4}{\sum_{u\in V(G)}\binom{d(u)}{2}}.
			\]
			By Jensen's inequality, 
			\[
				\sum_{x,y \in V(G)} \binom{|N(x) \cap N(y)|}{2} 
				\ge \binom{n}{2} \binom{\sum_{x,y \in V(G)} \binom{|N(x) \cap N(y)|}{2} / \binom{n}{2}}{2}.
			\]
			Noting that $\sum_{x,y\in V(G)}|N(x)\cap N(y)|=\sum_{u\in V(G)}\binom{d(u)}{2}$, we get
			\begin{align*}
				2c_4&=\sum_{x,y\in V(G)}\binom{|N(x)\cap N(y)|}{2}\\
				&\ge \binom{n}{2}\binom{\sum_{u\in V(G)} \binom{d(u)}{2}/\binom{n}{2}}{2}\\
				& \ge \frac{1}{3\binom{n}{2}}\left(\sum_{u\in V(G)}\binom{d(u)}{2}\right)^2,
			\end{align*}
			where the last inequality holds since $\sum_{u\in V(G)} \binom{d(u)}{2}/\binom{n}{2} \ge n\binom{d}{2}/\binom{n}{2}>3$ (using our assumption that $d>2\sqrt{n}$).
			Therefore,
			\[
				q \ge \frac{\sum_{u\in V(G)} \binom{d(u)}{2}}{3\binom{n}{2}} \ge \frac{n \binom{d}{2}}{3\binom{n}{2}} \ge \frac{d^2}{4n},
			\]
			where the second inequality follows from the assumption that $d(u) \ge d $ for every $u \in V(G)$, and the last inequality holds since $d > 2\sqrt{n}$. This finishes the proof.
		\end{proof}

	\subsection{Graphs with polynomially-large minimum degree} \label{sec:poly-large}

		In this subsection we prove that \cref{conj:log} holds for graphs with minimum degree at least $n^{\Omega(1)}$, which also proves the lower bound in \cref{thm:refined-degree}~\eqref{itm:sharp-medium-deg}. The following lemma appeared as~\cite[Theorem~3.4]{EKT}, and is a corollary of Johansson's theorem~\cite{Joh} for colouring triangle-free graphs and a connection between the fractional chromatic number and dense bipartite induced subgraphs (see \cite[Theorem 3.1]{EKT} and a discussion in \cref{sec:conclusion}).

		\begin{lem} \label{lem:log-regular}
			There is a constant $c>0$ such that any triangle-free graph with minimum degree $d$ and maximum degree $\Delta$ contains a bipartite induced subgraph of average degree at least $\frac{cd}{\Delta}\log \Delta$.
		\end{lem}

		Every graph with average degree $d$ has an induced subgraph with minimum degree at least $d/2$, so one can deduce a version of \cref{lem:log-regular} which guarantees a bipartite induced subgraph with \emph{minimum} degree at least $\frac{cd}{2\Delta}\log \Delta$.

		The other ingredient we will need is a theorem of Erd\H{o}s and Simonovits~\cite{ErdSim} (see also~\cite[Lemma~2.6]{JiaSei}).

		\begin{lem} \label{lem:regularization}
			Consider $\eps \in (0,1)$, and let $n$ be sufficiently large relative to $\eps$. Let $G$ be an $n$-vertex graph with $e(G) \ge n^{1+\eps}$. Then $G$ contains an induced subgraph $H$ on $m \ge n^{\eps(1-\eps)(1+\eps)}$ vertices such that $e(H) \ge \frac25 m^{1+\eps}$ and $\Delta(H)/\delta(H) \le 10 \cdot 2^{1/\eps^2+1}$.
		\end{lem}

		Now, consider $d=n^{\Omega(1)}$ and let $G$ be an $n$-vertex graph with minimum degree at least $d$. Then $G$ has $n^{1+\Omega(1)}$ edges, so we may apply \cref{lem:regularization} to deduce that it has an induced subgraph, all of whose degrees are in the range between $d'$ and $K d'$, for some $d' = n^{\Omega(1)}$ and some $K=O(1)$. By \cref{lem:log-regular}, inside this induced subgraph there is a bipartite induced subgraph with minimum degree at least $\frac{c}{2K}\log d'=\Omega(\log n)$, as desired.

\section{Reducing from $H$-free graphs to triangle-free graphs}\label{sec:reduction}

	In this section we show that in order to prove \cref{thm:log} it suffices to consider the case of triangle-free graphs.

	\begin{proof}[Proof of \cref{thm:log} under the assumption that it holds for $H = K_3$]
	
	    We note that it suffices to prove the theorem for $K_t$-free graphs, as every $H$-free graph is $K_{|H|}$-free.
        We proceed by induction. Fix $t>3$ and assume that for every $3\le s<t$, every $K_{s}$-free graph with minimum degree $q$ has an induced bipartite subgraph with
		minimum degree $\Omega\left(\log q/\log\log q\right)$.

		Let $G$ be a $K_t$-free graph on $n$ vertices with minimum degree at least $d$. As in \cref{sec:log}, we may assume that $G$ is $d$-degenerate, and fix an ordering of the vertices from left to right such that every vertex has at most $d$ neighbours to its right. Let $N^+(x)$ be the set of neighbours of a vertex $x$ to its right.

		Note that each $G\left[N^{+}\left(v\right)\right]$ is $K_{t-1}$-free (otherwise, appending $v$ to a copy of $K_{t-1}$ would give a copy of $K_{t}$ in $G$). We may assume that each $G\left[N^{+}\left(v\right)\right]$ has at most $d^{7/6}$ edges. Indeed, otherwise $G\left[N^{+}\left(v\right)\right]$ would have average degree at least $2d^{1/6}$, and therefore have an induced subgraph with minimum degree at least $d^{1/6}$. By the induction hypothesis it would therefore have an induced bipartite subgraph with minimum degree $\Omega\left(\log d^{1/6}/\log\log d^{1/6}\right)=\Omega\left(\log d/\log\log d\right)$.

		Now, let $U$ be a random subset of the vertices of $G$, containing each vertex with probability $p:=d^{-2/3}$, independently. Then, let $W\subseteq U$ be obtained by removing the leftmost vertex of every triangle in $G\left[U\right]$. From the definition of $W,$ we know that $G[W]$ is triangle-free. We wish to show that with positive probability $G\left[W\right]$ has at least $\left|W\right|d^{1/6}$ edges, which will imply that it has a subgraph with minimum degree at least $d^{1/6}$ and therefore has an induced bipartite subgraph with minimum degree $\Omega\left(\log d/\log\log d\right)$, by the induction hypothesis.

		Let $X_{1}$ be the number of edges in $G\left[U\right]$, let $X_{2}$ be the number of triangles in $G\left[U\right]$ and let $X_{3}$ be the number of edges $e$ in $G\left[U\right]$ which are incident to the leftmost vertex of a triangle in $G\left[U\right]$ (such that $e$ is not actually in the triangle). Then $G\left[W\right]$ has at least $X_{1}-2X_{2}-X_{3}$ edges.

		The number of edges in $G$ is at least $nd/2$ by the minimum degree assumption, and the number of triangles in $G$ is $\sum_{v\in V(G)}e\left(G\left[N^{+}\left(v\right)\right]\right) \le nd^{7/6}$ (counting triangles by their leftmost vertex). The number of edge-triangle pairs $(e,T)$ such that $e$ is not in $T$ and $e$ is incident to the leftmost vertex of $T$ is at most
		\[
			\sum_{v\in V(G)}d(v)e(G[N^{+}(v)]) \le \sum_{v\in V(G)}d(v)d^{7/6}=2e(G)d^{7/6} \le 2nd^{13/6}.
		\]
		Therefore
		\[
			\E\left|U\right|=np,\quad\E X_{1}\ge \tfrac12 ndp^{2},\quad\E X_{2}\le nd^{7/6}p^{3},\quad\E X_{3}\le 2nd^{13/6}p^{4},
		\]
		and
		\begin{align*}
		\E\left[\left(X_{1}-2X_{2}-X_{3}\right)-d^{1/6}\left|U\right|\right] & \ge \tfrac12 ndp^{2}-2nd^{7/6}p^{3}-2nd^{13/6}p^{4}-d^{1/6}np\\
		 & =n\left(\tfrac12 d^{-1/3}-2d^{-5/6}-2d^{-1/2}-d^{-1/2}\right)>0
		\end{align*}
		for sufficiently large $d$. It follows that there is an outcome of $U$ with $\left(X_{1}-2X_{2}-X_{3}\right)-d^{1/6}\left|U\right|>0$. Then, $G\left[W\right]$ is an induced subgraph with at least $X_{1}-2X_{2}-X_{3}> d^{1/6}\left|U\right|\ge d^{1/6}\left|W\right|$ edges, as desired.
	\end{proof}

\section{Upper bounds on minimum degree}\label{sec:upper}
	In this section we describe both probabilistic and explicit constructions of triangle-free graphs which have no high-degree bipartite induced subgraphs. These will be used for the proofs of the upper bounds in \cref{thm:refined-degree}. To be specific, the upper bounds for \eqref{itm:medium-deg}-\eqref{itm:sharp-medium-deg} are proved in \cref{sec:probabilistic}, and the upper bound for \eqref{itm:shart-large-deg} is proved in \cref{sec:explicit}.

	\subsection{Probabilistic construction}\label{sec:probabilistic}

		In this subsection we prove the upper bounds in \cref{thm:refined-degree}~\eqref{itm:medium-deg}-\eqref{itm:sharp-medium-deg}. This essentially boils down to the following fact.

		\begin{lem} \label{lem:prob-construction}
			For any $n$, there exists a triangle-free graph $G_n$ on between $n/2$ and $n$ vertices such that every vertex of $G_n$ has degree $\Theta(\sqrt{n})$, and every induced bipartite subgraph of $G$ has minimum degree at most $O\left(\log n\right)$.
		\end{lem}

		\begin{proof}[Proof of the upper bounds in \cref{thm:refined-degree}~\eqref{itm:medium-deg}-\eqref{itm:sharp-medium-deg} given \cref{lem:prob-construction}]
			Let $G_n$ be a graph as in the statement of \cref{lem:prob-construction}. By blowing up each of the vertices into an independent set of size $1$ or $2$ (as in the proof of \cref{thm:ManyEdges}), in such a way that the total number of verties in the blown-up graph is $n$, we obtain a graph on $n$ vertices, with minimum degree $\Theta(\sqrt{n})$, whose every induced bipartite subgraph has minimum degree $O\left(\log n\right)$. By replacing $G_n$ with the blown-up graph, we may thus assume that $G_n$ itself has exactly $n$ vertices.

			Let $\delta(G_{n})=\Theta(\sqrt n)$ be the minimum degree of the graph $G_n$. We divide our analysis into two cases: $\delta(G_{n}) \le d \le n/2$, and $d \le \delta(G_n)$.

			Suppose that $\delta(G_{n}) \le d \le n/2$. We will show that $g(n,d)=O((d^2\log d)/n)$ by taking blowups as in the proof of \cref{thm:ManyEdges}. Let $n'$ be maximal such that $n'\le n$ and $\floor{n/n'} \delta(G_{n'})\ge d$. (There exists such an $n'$ because $n' = 2$ satisfies the required inequalities.) We claim that $n' = \Theta( (n/d)^2)$. Indeed, recall that $\delta(G_{n'}) = \Theta(\sqrt{n'})$. We thus have $d = O(n / \sqrt{n'})$, and so $n' = O( (n/d)^2)$. For the lower bound, note that if $n'' = \floor{c (n/d)^2}$, for a sufficiently small constant $c > 0$, then $\floor{n / n''} \delta(G_{n''}) \ge d$. It follows that $n' \ge n'' = \Omega((n/d)^2)$.
			Blow up the vertices of $G_{n'}$ into independent sets with almost-equal sizes (about $n/n'$) to obtain an $n$-vertex graph with minimum degree at least $d$, in which every induced bipartite subgraph has minimum degree at most $O((n/n')\log (n'))=O((d^2\log d)/n)$. This shows that $g(n,d)=O((d^2\log d)/n)$ in the regime where $\delta(G_n) \le d \le n/2$.

			Finally, we consider the case $d \le \delta(G_n)$. Let $n'$ be minimal such that $\delta(G_{n'}) \ge d$, so $n' = \Theta(d^2)$. Consider the graph which is a disjoint union of $\floor{n/n'}$ copies of $G_{n'}$. This graph has between $n/2$ and $n$ vertices, it has minimum degree at least $d$, and all of its induced bipartite subgraphs have minimum degree $O(\log n') = O(\log d)$ (because every induced bipartite subgraph is a disjoint union of induced bipartite subgraphs of the copies of $G_{n'}$). As above, by blowing up each vertex into an independent set of size $1$ or $2$, we can obtain a graph on exactly $n$ vertices, with minimum degree at least $d$, whose every induced bipartite subgraph has minimum degree $O(\log d)$. It follows that $g(n,d) = O(\log d)$ when $d \le \delta(G_n)$.
		\end{proof}

		To prove \cref{lem:prob-construction} we will follow the approach of Krivelevich~\cite{Kri95}. Our main tools are the following large deviation inequalities.

		\begin{lem} \label{lem:Krivelevich}
			The following hold.
			\begin{enumerate}[\rm (a)]
				\item \label{itm:chernoff}
					Let $n \in \NN$ and $p,\eps \in [0,1]$. Then $\Pr(|\Bin(n,p)-np| \ge \eps np) \le 2 \cdot e^{-\eps^2 np/3}$.
				\item \label{itm:krivelevich}
					Consider a finite set $\Gamma$. Let $\{X_i:i \in \Gamma\}$ be a set of independent random variables each supported on $\{0,1\}$, and let $\cF \subset 2^{\Gamma}$ be a collection of subsets of $\Gamma$. Given $F\in \cF$, write $X_F$ for the random variable $\prod_{i \in F}X_i$. Now, define $X=\sum_{F\in \cF}X_F$,
					and
					\[
						X_0=\max\{m: \enskip \exists \enskip \text{pairwise disjoint sets $F_1,\ldots,F_m\in \mathcal{F}$ such that $X_{F_1}=\ldots=X_{F_m}=1$}\}.
					\]
					Then we have $\Pr(X_0 \ge 5 \Ex X)  \le e^{-\Ex X}$. 
			\end{enumerate}
		\end{lem}

		Part \eqref{itm:chernoff} of \cref{lem:Krivelevich} is the well-known Chernoff bound (see, e.g., \cite[Corollary 21.7]{FK16}), while Part \eqref{itm:krivelevich} follows from~\cite[Claim~1]{Kri95}. We will use \cref{lem:Krivelevich} to prove the following statement about triangle-free subgraphs of a random graph.
		 
		\begin{lem}\label{lem:key}
			Consider the binomial random graph $G\sim G(n,p)$, with $p=cn^{-1/2}$ for some absolute constant $c \in (0,\tfrac{1}{20})$. Set $a=\max \{10^6c^{-1},12c^{-3},28\}$. Then w.h.p.\ $G$ contains a triangle-free subgraph $H$ satisfying the following properties.
			\begin{align}
				\Delta(H)& \le 1.01np, \label{eq:deg-I}\\
				e_H(A,B)&\le a|A|\log n \enskip \text{for every disjoint sets} \enskip A,B \subset V(G) \enskip
				\text{with} \enskip |A|=|B|\le a\sqrt{n}\log n, \label{eq:bipartite-I}\\
				e(H)&\ge 0.9 p \binom{n}{2},\label{eq:edges}\\
				\alpha(H)& \le a\sqrt{n} \log n.\label{eq:ind}
			\end{align}
		\end{lem}

		Before proving \cref{lem:key}, we show how to deduce \cref{lem:prob-construction} from it.

		\begin{proof}[Proof of \cref{lem:prob-construction}]
			Let $H$ be a triangle-free graph which satisfies \eqref{eq:deg-I}--\eqref{eq:ind}. Write $H'$ for the graph obtained from $H$ by iteratively removing vertices of degree at most $pn/30$. Obviously $H'$ is triangle-free. We have $e(H') \ge e(H)-pn^2/30\ge 0.83p\binom{n}{2}$ by \eqref{eq:edges}. Combining this with \eqref{eq:deg-I} we deduce that $|H'|\in [0.8n,n]$, and each vertex of $H'$ has degree $\Theta(pn)=\Theta(\sqrt{n})$. Moreover, from \eqref{eq:bipartite-I} and \eqref{eq:ind} we learn that every induced bipartite subgraph of $H'$ has minimum degree at most $O(\log n)$.
		\end{proof}

		\begin{proof}[Proof of \cref{lem:key}]
			We follow the method of Krivelevich in~\cite{Kri95}. Let $\cT$ be any maximal family of edge-disjoint triangles in $G$. Let $H$ be the triangle-free graph obtained from $G$ by removing all edges in $\cT$. We will use \cref{lem:Krivelevich} to show that w.h.p.\ $H$ satisfies \eqref{eq:deg-I}--\eqref{eq:ind}. For ease of notation, set 
			$k=\floor{a\sqrt{n}\log n}$.

			As $d_G(v)\sim \Bin(n-1,p)$ for every $v \in V(G)$, using \cref{lem:Krivelevich} (a) and the union bound we get 
			\[
				\Pr\left(\exists v \in V(G) \text{ with } d_G(v) \ge 1.01np\right) 
				\le n \cdot 2e^{-\Omega(np)}=n \cdot e^{-\Omega(\sqrt{n})}=o(1),
			\]
			Thus w.h.p.\ $d_G(v) \le 1.01np$ for every $v \in V(G)$. This implies that \eqref{eq:deg-I} holds w.h.p. 

			Since $e_G(A,B) \sim \Bin(|A||B|,p)$ for every pair of disjoint sets $A, B \subseteq V(G)$, applying \cref{lem:Krivelevich} (a) and the union bound we find that
			\begin{align*}
				& \Pr\left(\exists \enskip \text{disjoint vertex sets $A,B$ with $|A|=|B| \le k$ and $e_G(A,B) \ge a|A|\log n$}\right)\\ 
				& \qquad \le \sum_{\ell=1}^{k} n^{2\ell} \cdot \Pr\left(\Bin(\ell^2,p) \ge a\ell\log n\right)\\
				& \qquad \le \sum_{\ell=1}^{k} n^{2\ell} \cdot \Pr\left(\Bin\left(\floor{\tfrac12 ap^{-1}\ell\log n},p\right) \ge a\ell\log n\right)\\
				& \qquad \le \sum_{\ell=1}^{k} n^{2\ell} \cdot 2 e^{-(a\ell\log n)/7} \le 2\sum_{\ell=1}^{k}n^{-2\ell}=o(1). %\label{eq:pair-distribution}
			\end{align*}
			The second inequality holds since $\ell^2 \le \tfrac12 ap^{-1}\ell\log n$ for $\ell \le a\sqrt{n}\log n$ and $p \le \frac{1}{20}n^{-1/2}$, while the last inequality follows from the assumption that $a \ge 28$. We have proved that  \eqref{eq:bipartite-I} holds w.h.p.
			It remains to consider \eqref{eq:edges} and \eqref{eq:ind}.

			For every subset $S\subset V(G)$ of size $k$, denote by $Y_S$ the number of triangles that have at least two vertices in $S$, and by $Z_S$ the maximum number of pairwise edge-disjoint triangles with at least two vertices in $S$. Clearly $Z_S \le Y_S$.  
			We claim that w.h.p.
			\begin{equation}\label{eq:key}
			e_G(S)\ge 34\,Z_ S \text{ for every subset $S\subset V(G)$ of size $k$.}
			\end{equation}

			Before proving \eqref{eq:key}, we will show how it implies \eqref{eq:edges} and \eqref{eq:ind}. Since $e_G(S) \sim \Bin(\binom{k}{2},p)$ for every subset $S\subset V(G)$ of size $k$, using \cref{lem:Krivelevich}~\eqref{itm:chernoff} and the union bound we get that
			\[
				\Pr\left(\exists \enskip \text{a vertex set $S$ of size $k$ with $e_G(S) \le 0.99 p \binom{k}{2}$}\right) \le n^k\cdot 2 e^{-10^{-5}p\binom{k}{2}} \le 2n^{-k}=o(1).
			\]
			In the second inequality we used the facts that $p=cn^{-1/2},$ $k=\floor{a\sqrt{n}\log n}$ and $a \ge 10^6 c^{-1}$. Hence w.h.p.\ every size-$k$ vertex set $S$ satisfies
			\begin{equation}\label{eq:edge-S}
				e_G(S) \ge 0.99p\binom{k}{2}.
			\end{equation}
			Now, from the definition of $H$ and $Z_S$ we see that w.h.p.\ every size-$k$ subset $S \subset V(G)$ spans at least 
			\begin{equation*}
				e_G(S)-3Z_S\overset{\eqref{eq:key}}{>}0.91 e_G(S) \overset{\eqref{eq:edge-S}}{>} 0.9 p\binom{k}{2}
			\end{equation*}
			edges in $H$. It follows that w.h.p.\ $\alpha(H) \le k$ and $e(H)\ge 0.9 p \binom{n}{2}$, as required.

			We now return to the proof of \eqref{eq:key}. Given a size-$k$ vertex set $S$, we will bound $\Pr(e_G(S)< 34 Z_S)$ using \cref{lem:Krivelevich}~\eqref{itm:krivelevich}. 
			Note that 
			$\frac{\Ex [e_G(S)]}{\Ex Y_S} \ge \frac{p\binom{k}{2}}{n\binom{k}{2}p^3} \ge 400$,
			assuming $p \le \frac{1}{20}n^{-1/2}$. 
			Hence 
			\begin{align*}
				\Pr(e_G(S) < 34 Z_S) & \le \Pr(e_G(S) \le \Ex [e_G(S)]/2)+\Pr(34Z_S \ge \Ex [e_G(S)]/2)\\
				& \le \Pr(e_G(S) \le \Ex [e_G(S)]/2)+\Pr(Z_S \ge 5\Ex Y_S).
			\end{align*}
			Since $e_G(S) \sim \Bin(\binom{k}{2},p)$, applying \cref{lem:Krivelevich}~\eqref{itm:chernoff} we get that $\Pr(e_G(S) \le \Ex [e_G(S)]/2) \le 2 e^{-p\binom{k}{2}/12}$. Moreover, \cref{lem:Krivelevich}~\eqref{itm:krivelevich} implies that $\Pr(Z_S \ge 5\Ex Y_S) \le e^{-\Ex Y_S}= e^{-(1+o(1))n\binom{k}{2}p^3}$.
			Therefore, using the facts $p=cn^{-1/2} \le \frac{1}{20}n^{-1/2}$ and $k=\omega(1)$ we obtain 
			\[
				\Pr(e_G(S) < 34 Z_S) \le 2e^{-p\binom{k}{2}/12}+e^{-(1+o(1))n\binom{k}{2}p^3} \le 3 e^{-\frac{c^2}{3}pk^2}.
			\]
			With the union bound, we deduce that
			\[
				\Pr(\text{\eqref{eq:key} does not hold}) \le n^{k}\cdot 3e^{-\frac{c^2}{3}pk^2} \le 3e^{k \log n-\frac{c^2}{3}pk^2}\le  3e^{-k \log n}=o(1),
			\]
			where the last inequality holds since $p=cn^{-1/2}$, $k=\floor{a\sqrt{n}\log n}$ and $a\ge 12c^{-3}$. This finishes our proof of \cref{lem:key}.
		\end{proof}

	\subsection{Explicit construction}\label{sec:explicit}

		In this section we prove the upper bound in \cref{thm:refined-degree}~\eqref{itm:shart-large-deg}, by blowing up a sequence of triangle-free quasirandom graphs constructed by Alon~\cite{Alo94}. A graph $G$ is called an $(n,d,\lambda)$-graph if it is $d$-regular, has $n$ vertices, and all eigenvalues of its adjacency matrix, save the largest, are smaller in absolute value than $\lambda$. Alon's construction implies the following lemma.

		\begin{lem} \label{lem:alon}
			For every integer $n$ of the form $n=2^{3k}$ for some $k\in \mathbb{N}$, one can explicitly construct a triangle-free $(n,d,\lambda)$-graph $G_{n}$ with $d=2^{2k-2}-2^{k-1}=\Theta(n^{2/3})$ and $\lambda=O(n^{1/3})$.
		\end{lem}

		The expander mixing lemma (see, for example, \cite[Theorem 2.11]{KB06}) asserts that $(n,d,\lambda)$-graphs with small $\lambda$ have low discrepancy, as follows.

		\begin{lem}[Expander mixing lemma]
		Suppose $G$ is an $(n,d,\lambda)$-graph. Then for any disjoint vertex sets $A,B\subset V(G)$, we have
		\[
		\left|e(A,B)-\frac{d}{n}|A||B| \right| \le \lambda \sqrt{|A||B|}.
		\]
		\end{lem}

		In much the same way as we used \cref{thm:GuoWarnke} in the proof of \cref{thm:ManyEdges}, we can use the expander mixing lemma to prove that the graphs in \cref{lem:alon} do not have bipartite induced subgraphs with high minimum degree.

		\begin{cor}
		Let $n=2^{3k}$, and let $G_n$ be a graph coming from \cref{lem:alon}. Then every induced bipartite subgraph of $G_n$ has minimum degree at most $O(n^{1/3})$.
		\end{cor}

		\begin{proof}
			From the expander mixing lemma, we learn that $\alpha(G_n) \le 2(n\lambda/d + 1) = O(n^{2/3})$. Moreover, for any two sets $A,B\subset V(G_n)$ of size $O(n^{2/3})$, the expander mixing lemma gives
			\[
				e(A,B) \le n^{-1/3}|A||B|+ n^{1/3}\sqrt{|A||B|}=O(n^{1/3}\max{\{|A|,|B|\}}).
			\]
			Thus any induced bipartite subgraph of $G_n$ has minimum degree at most $O(n^{1/3})=O(d^2/n)$.
		\end{proof}

		Now, we can prove the upper bound in \cref{thm:refined-degree}~\eqref{itm:shart-large-deg} by taking blowups as in previous proofs.
		
		\begin{proof}[Proof of the upper bound in \Cref{thm:refined-degree}~\eqref{itm:shart-large-deg}]
			Note that the upper bound holds trivially when $d \ge n / 32$. Thus, we may assume that $n^{2/3} \le d \le n/32$. Let $k$ be maximal such that $2^{3k} \le n$ and $\floor{n/2^{3k}} d(G_{2^{3k}}) \ge d$, where $G_{2^{3k}}$ is the graph given by \cref{lem:alon}. We claim that such $k$ exists. Indeed, let $\ell$ be an integer such that $n / (32 d) \le 2^{\ell} \le n / (16d)$. Then $2^{3\ell} \le (n / (16d))^3 \le n/2$, and
			\[
				\floor{\frac{n}{2^{3\ell}}} d(G_{2^{3\ell}}) 
				\ge \frac{n}{2^{3\ell +1}}\left(2^{2\ell-2} - 2^{\ell-1}\right) 
				= \frac{n}{2^{\ell}}\left(\frac{1}{8} - \frac{1}{4 \cdot 2^{\ell}}\right)
				\ge \frac{n}{16 \cdot 2^{\ell}} 
				\ge d.
			\]
			This implies that there exists $k$ that satisfies both inequalities, as claimed.
			As $d(G_{2^{3k}})=2^{2k-2}-2^{k-1}=\Theta((2^{3k})^{2/3})$, this means that $n':=2^{3k}=\Theta((n/d)^3)$. Blow up the vertices of $G_{n'}$ into independent sets with almost-equal sizes (about $n/n'$) to obtain an $n$-vertex graph with minimum degree at least $d$, in which every induced bipartite subgraph has minimum degree at most $O((n/n')(n')^{1/3})=O(d^2/n)$.
		\end{proof}

\section{Concluding remarks} \label{sec:conclusion}

	In this paper we have proved that for every fixed graph $H$, any $H$-free graph with minimum degree $d$ contains an induced bipartite subgraph of minimum degree at least $\Omega\left(\log d/\log \log d\right)$. It would be very interesting to improve this to $\Omega\left(\log d\right)$, and thus fully confirm \cref{conj:log} of Esperet, Kang and Thomass\'e. We note that by the methods we employed in \cref{sec:reduction}, an improvement to $\Omega\left(\log d\right)$ when $H$ is a triangle, would imply the same for general $H$.

	Given a fixed graph $H$, let $g_H(n,d)$ be the maximum $g$ such that every $H$-free graph with $n$ vertices and minimum degree at least $d$ contains an induced bipartite subgraph with minimum degree at least $g$. When $H=K_3$, we gave quite accurate estimates on $g_{H}(n,d)$. It would be interesting to study this function further for other forbidden subgraphs $H$. Some of our methods can be generalised, but the overall picture is not clear, since the behaviour of $g_H(n,d)$ seems closely related to the (hard) Ramsey problem of bounding independence number of $H$-free graphs. See \cite[Section 6]{BVKP} for more explicit problems and conjectures regarding the function $g_H(n, d)$.
		
	There is an interesting connection between the existence of induced bipartite graphs with large minimum degree and the \emph{fractional chromatic number}. Recall that a \emph{fractional colouring} of a graph $G$ is an assignment of non-negative weights to the independent sets of $G$, in such a way that for each vertex $u$, the sum of weights of independent sets that contain $u$ is at least $1$. The \emph{fractional chromatic number} of $G$ is the minimum sum of weights of independent sets, over all possible fractional colourings. Esperet, Kang and Thomass\'e \cite[Theorem 3.1]{EKT} proved that a graph with minimum degree $d$ and fractional chromatic number at most $k$ has a bipartite induced subgraph of average degree at least $d/k$. Cames van Batenburg, de Joannis de Verclos, Kang and Pirot \cite{BVKP} exploited this connection in order to prove \Cref{thm:refined-degree} \eqref{itm:large-deg} by proving an upper bound on the fractional chromatic number of triangle-free graphs  on $n$ vertices with minimum degree $d$, which is tight when $d$ is sufficiently large with respect to $n$. It is plausible that a similar approach could be used to obtain an alternative proof of \Cref{thm:refined-degree} \eqref{itm:medium-deg}, or even to improve our bound. This raises the following question: how large can the fractional chromatic number of a $d$-degenerate triangle-free graph be? (One can assume that the graph is $d$-degenerate, as in \Cref{sec:lower}.) More precisely, is it true that the fractional chromatic number of such a graph is at most $O(d / \log d)$? Harris \cite{Ha} conjectured that the answer to the latter question is `yes'. An affirmative answer to this question would be tight and would prove \Cref{conj:log}. We note that under the stronger assumption that the graph has maximum degree $d$ it was proved by Johansson \cite{Joh} (see also Molloy \cite{Mo}) that already the chromatic number is bounded by $O(d/\log d)$. On the other hand, this is no longer the case for $d$-degenerate graphs, as shown in \cite{AKS1}.

\subsection*{Acknowledgements}
 
	We would like to thank the anonymous referees for their helpful comments and suggestions.

\end{document}